\documentclass[reqno0,12pt]{amsart}
\pdfoutput=1
\usepackage{dsfont}
\usepackage{bbm}
\usepackage[nodate]{datetime}
\usepackage[hmargin=20mm,top=28mm,bottom=28mm,a4paper]{geometry}
\usepackage{color}
\usepackage{subfig}
\usepackage{fancyhdr}
\usepackage{amsmath,amssymb,amsfonts,amsthm}
\usepackage{amstext}
\usepackage{amsmath}
\usepackage{amssymb}
\usepackage{enumerate}
\usepackage{amsbsy}
\usepackage{amsopn}
\usepackage{bbm,amsthm}
\usepackage{amscd}
\usepackage{amsxtra}
\usepackage{todonotes}

\newtheorem{theorem}{Theorem}[section]
\newtheorem*{theorem*}{Theorem}
\newtheorem{lemma}{Lemma}[section]
\newtheorem{proposition}{Proposition}[section]

\newtheorem{conjecture}{Conjecture}[section]

\theoremstyle{remark}

\newcommand{\CC}{\mathds{C}}

\newcommand{\RR}{\mathds{R}}
\newcommand{\NN}{\mathds{N}}

\newcommand{\ind}{\mathds{1}}

\newcommand{\norma}{\bigg{\|}}

\newcommand{\EE}{\mathbb{E}}

\newcommand{\MM}{\mathcal{M}}

\DeclareMathOperator{\lcm}{lcm}
\DeclareMathOperator{\rad}{rad}

\begin{document}
\title{The Erd\H{o}s discrepancy problem over the squarefree and cubefree integers}
\author{Marco Aymone}
\dedicatory{In memory of Vladas Sidoravicius}
\begin{abstract}
 Let $g:\mathbb{N}\to\{-1,1\}$ be a completely multiplicative function, $\mu$ be the M\"obius function and $\mu_2^2(n)$ be the indicator that $n$ is cubefree. We prove that $f=\mu^2g$ and $f=\mu_2^2g$ have unbounded partial sums. Our proofs are built upon Klurman and Mangerel's proof of Chudakov's conjecture, Klurman's work on correlations of multiplicative functions and Tao's resolution of the Erd\H{o}s discrepancy problem.
\end{abstract}

\maketitle

\section{Introduction.}
The Erd\H{o}s discrepancy problem asks if there exists an arithmetic function $f:\NN\to\{-1,1\}$ that it is well distributed along all homogeneous arithmetic progressions, in the sense that
\begin{equation*}
\sup_{x,d}\bigg{|}\sum_{n\leq x}f(nd)\bigg{|}<\infty.
\end{equation*}

This problem became the subject of the Polymath5 project in 2010, and as Tao showed in 2015 \cite{taodiscrepancy}, for such $f$ the $\sup$ of the quantity above is equal to infinity. In particular, if $f:\NN\to\{-1,1\}$ is completely multiplicative, then $f$ has unbounded partial sums. However, if $f:\NN\to\{-1,1\}$ is assumed to be only multiplicative, then $f$ may have bounded partial sums, for instance $f(n)=(-1)^{n+1}$.

Following Granville and Soundararajan \cite{granvillepretentious}, define the ``distance'' up to $x$ between two multiplicative functions $|f|,|g|\leq 1$ as
\begin{equation*}
\mathbb{D}(f,g;x):=\left(\sum_{p\leq x}\frac{1-Re(f(p)\overline{g(p))}}{p}\right)^{1/2},
\end{equation*}
where in the sum above $p$ denotes a generic prime. Moreover, let $\mathbb{D}(f,g):=\mathbb{D}(f,g;\infty)$, and we say that $f$ pretends to be $g$, or that $f$ is $g$-pretentious if $\mathbb{D}(f,g)<\infty$.

In \cite{taodiscrepancy}, Tao showed that if $f:\NN\to\{-1,1\}$ is multiplicative and has bounded partial sums, then $f$ is $1$-pretentious and $f(2^k)=-1$ for all $k\geq 1$. In \cite{klurmancorrelation}, Klurman provided a complete classification of such $f$ with bounded partial sums by proving the Erd\H{o}s-Coons-Tao conjecture:  A multiplicative function $f:\NN\to\{-1,1\}$ has bounded partial sums if and only if  there exists an integer $m\geq 1$ such that for all $n\in\NN$, $f(n+m)=f(n)$ and $\sum_{n=1}^mf(n)=0$.

When we allow zero values, that is, for a completely multiplicative function $f:\NN\to\{-1,0,1\}$ such that $f(p)=0$ only for a finite subset of primes $p$, then also $f$ may have bounded partial sums, for instance, any real and non-principal Dirichlet character $\chi$. Moreover, in \cite{klurmanchudakov}, Klurman and Mangerel established Chudakov's conjecture: Assume that $f:\NN\to\CC$ is completely multiplicative, takes only a finite number of values, has bounded partial sums and the cardinality $|\{p\mbox{ is a prime}:f(p)=0\}|<\infty$. Then $f$ is a Dirichlet character.

Here we are interested in the discrepancy problem for multiplicative functions $f$ assuming zero values and such that at primes $f(p)=\pm 1$. More precisely, we are interested in knowing if such $f$ has bounded, or unbounded partial sums when the summation $\sum_{n\leq x}' f(n)$ is restricted to certain subsets of integers with additional arithmetic properties. As shown in \cite{klurmancorrelation}, \cite{klurmanchudakov} and \cite{taodiscrepancy} (see Proposition \ref{proposicao chi pretentious} below), if such $f$ has bounded partial sums, then $f$ must be $\chi$-pretentious for some real and primitive Dirichlet character $\chi$.

In \cite{aymoneresemblingmobius} the author addressed the question of how small we can make the partial sums of a multiplicative function supported on the squarefree integers, that is, $f=\mu^2g$, where $g:\NN\to\{-1,1\}$ is a multiplicative function and $\mu$ is the M\"obius function. As mentioned above, when we look at multiplicative functions with small partial sums, then one needs to look first at multiplicative functions that pretend to be a Dirichlet character $\chi$. It has been proved that if $f$ is \textit{strongly} $\chi$-pretentious for some real and non-principal Dirichlet character $\chi$, in the sense that $\sum_{p\leq x}|1-f(p)\chi(p)|\ll 1$, then the partial sums of $f$ cannot have too much cancellation, more precisely $\sum_{n\leq x}f(n)$ is $\Omega(x^{1/4-\epsilon})$. This motivates us to conjecture:
\begin{conjecture}\label{conjecture squarefree} If $g:\NN\to\{-1,1\}$ is multiplicative, then $f=\mu^2g$ is such that $\sum_{n\leq x}f(n)=\Omega(x^{1/4-\epsilon})$.
\end{conjecture}
Towards this conjecture, our first result states:
\begin{theorem}\label{teorema correlacoes resembling mobius} Let $f=\mu^2g$ where $g:\NN\to\{-1,1\}$ is multiplicative. Then $f$ has unbounded partial sums.
\end{theorem}
 
 Some partial results in the direction of Theorem \ref{conjecture squarefree} were previously obtained by the author in a preprint version of this paper, and later by Klurman et al. \cite{klurmanteravainen} in their preprint (arXiv, version 1). In a newer version of \cite{klurmanteravainen} (to appear in Trans. of the Amer. Math. Soc.), Theorem \ref{teorema correlacoes resembling mobius} has been settled by using \textit{the rotation trick}, a new method discovered by the authors of \cite{klurmanteravainen}. Here we present a different proof inspired by the proof of Chudakov's conjecture \cite{klurmanchudakov}. As mentioned above, if our $f$ has bounded partial sums, then $f$ is $\chi$-pretentious for some real and primitive Dirichlet character $\chi$ of conductor $q$, and when this happens it has been established \cite{klurmancorrelation} that (see Theorem \ref{teorema correlacoes klurman} below):
\begin{equation*}
S_d:=\lim_{x\to\infty}\frac{1}{x}\sum_{n\leq x}f(n)\overline{f(n+d)}:=v(d;f\chi)\lambda(d;|f|,\chi).
\end{equation*}
The function $\lambda(d; |f|, \chi)$ depends essentially on the sum $\sum_{a\leq q }\chi(a)\overline{\chi(a+d)}$, and the function $v(d)$ can be expressed as a constant $C=C(f)$ times $1\ast u(d)$, where $u$ is a well behaved multiplicative function. With these correlations formulae, we can give precise estimates for
\begin{equation}\label{equacao definicao Lambda introducao}
\Lambda(H):=\lim_{x\to\infty}\frac{1}{x}\sum_{n\leq x}\bigg{|}\sum_{k=n+1}^{n+H} f(k) \bigg{|}^2=\sum_{|h|\leq H}(H-|h|)S_{|h|}.
\end{equation}
If $f$ has bounded partial sums, then $\Lambda(H)=O(1)$, and as we show below, the structure of $u$ implies that this is impossible in the squarefree case. 

The next natural question is to ask if a result analogous to Theorem \ref{teorema correlacoes resembling mobius} holds in the $k$-free integers. In the cubefree case we have:

\begin{theorem}\label{teorema cube free} Let $g:\NN\to\{-1,1\}$ be completely multiplicative and let $f=\mu_2^2g$, where $\mu_2^2(n)$ is the indicator that $n$ is \textit{cubefree}. Then $f$ has unbounded partial sums.
\end{theorem}

More generally, from the analysis presented in this paper, the interested reader can outline the details and prove that if $f=\mu_k(n)^2g$, where $\mu_k(n)^2$ is the indicator that $n$ is $k$-free and $g:\NN\to\{-1,1\}$ is a completely multiplicative function, then $f$ will have unbounded partial sums for all $k\geq 2$. The proof of this follows the lines of the proof of Theorem \ref{teorema correlacoes resembling mobius} if $k$ is even and the lines of the proof of Theorem \ref{teorema cube free} if $k$ is odd. The hard case is when $g$ is only multiplicative, which is the subject of future research. Further, we speculate that similarly to Conjecture \ref{conjecture squarefree}, if $f=\mu_k(n)^2g$, then $\sum_{n\leq x}f(n)=\Omega(x^{1/(2k)-\epsilon}) $, see the discussion in the introduction of \cite{aymonekfree}.

Let $\chi$ be a real and primitive Dirichlet character of conductor $q$ and let $\chi^*$ be a completely multiplicative extension of $\chi$, \textit{i.e.}, $\chi^*$ is completely multiplicative, $\chi^*(n)=\chi(n)$, if $\gcd(n,q)=1$, and at each prime $p|q$, $\chi^*(p)=\pm1$. Let $f_1=\mu^2\chi^*$ and
$f_2=\mu_2^2\chi^*$. Clearly $f_1$ and $f_2$ are $\chi$-pretentious. In \cite{aymoneresemblingmobius}, it has been proved that if $\chi$ is real and non-principal, and if we assume RH for $L(s,\chi)$, then the statement $\sum_{n\leq x}f_1(n)\ll x^{1/4+\epsilon}$ for any $\epsilon>0$ implies RH for $\zeta$. Towards this fact we have:
\begin{theorem}\label{teorema H^1/4 cancelacao} Let $f_1$ and $f_2$ be as above and let $\Lambda(H)$ be as in (\ref{equacao definicao Lambda introducao}). Then, in the squarefree case, $\Lambda(H)\ll H^{1/2}$; in the cubefree case $\Lambda(H)\ll H^{1/3}$.
\end{theorem}
We can interpret $\frac{1}{x}\sum_{n\leq x} |\cdot|^2$ as an expectation $\EE |\cdot|^2$ where $n$ is a random number chosen with probability $\frac{1}{x}$ in the set $\{1,...,x\}\subset \NN$. Thus, for each fixed $x$, we have H\"{o}lder's inequality: $\EE |\cdot|\leq (\EE |\cdot|^2)^{1/2}$. By defining $\EE^*=\limsup_{x\to\infty} \EE$, Theorem \ref{teorema H^1/4 cancelacao} says that, in the squarefree case, for a certain proportion of $n\in\NN$, the short sum $|\sum_{k=n+1}^{n+H} f_1(k)|$ has $H^{1/4}$-cancellation, and similarly, we have $H^{1/6}$- cancellation in the cubefree case, which is, as mentioned above, consistent with GRH. It seems likely that we will have the same pattern if $f_k$ is the restriction of $\chi^*$ to the $k$-free integers.

\noindent \textbf{Acknowledgements.}  I would like to thank Oleksiy Klurman for fruitful email exchanges on this subject and for kindly explaining his (and Sacha Mangerel's) solution of Chudakov's conjecture. Also, I would like to thank Winston Heap and the anonymous referee for several remarks, corrections and useful suggestions. 

This paper is dedicated to the memory of my Phd advisor, Vladas Sidoravicius. I would like to thank Vladas for his great intuition, ideas and enthusiasm that played a fundamental role in my academic trajectory.

\section{Proofs of the main results}
\subsection*{Notation} Here $\mathbb{U}=\{z\in\CC:|z|\leq 1\}$. We use both $f(x)\ll g(x)$ and $f(x)=O(g(x))$ whenever there exists a constant $C>0$ such that for all  $x\geq1$ we have that $|f(x)|\leq C|g(x)|$. Further, $\ll_\delta$ means that the implicit constant may depend on $\delta$. The standard $f(x)=o(g(x))$ means that $\lim_{x\to\infty}\frac{f(x)}{g(x)}=0$. We say that $f(x)=\Omega(g(x))$ if $\limsup_{x\to\infty}\frac{|f(x)|}{g(x)}>0$. We let $\mathcal{P}$ for the set of primes and $p$ for a generic element of $\mathcal{P}$. The notation $p^k\| n$ means that $k$ is the largest power of $p$ for which $p^k$ divides $n$. The M\"obius function is denoted by $\mu$, \textit{i.e.}, the multiplicative function with support on the squarefree integers and such that at the primes $\mu(p)=-1$. Dirichlet convolution is denoted by $\ast$. Given a subset $A\subset\NN$, we denote by $\ind_A(n)$ the characteristic function of $A$. We let $\mu_2^2(n)=\ind[n\mbox{ is cubefree}]$, \textit{i.e.}, at power of primes $\mu_2^2(p^l)=\ind_{l\in\{0,1,2\}}$. Here $\rad(n)=\prod_{p|n}p$. We let $\NN(q)=\{m\in\NN:\rad(m)|q\}$. Finally, $\omega(k)$ is the number of distinct primes that divide a certain $k$. 

\subsection{Preliminaries} The proof of our results starts with the Proposition below, which is essentially due to Tao \cite{taodiscrepancy} Remark 3.1, Klurman \cite{klurmancorrelation} Lemma 4.3 and Klurman and Mangerel \cite{klurmanchudakov} Lemma 5.1. Hence, for the convenience of the reader, in Appendix \ref{appendix proposicao chi pretentious}, we only indicate how to obtain it from these results.
\begin{proposition}\label{proposicao chi pretentious} Suppose that $f:\NN\to\{-1,0,1\}$ is a multiplicative function such that $\sum_{n\leq x}f^2(n)=(c+o(1))x$ for some positive constant $c$, and further assume that at primes $f(p)^2=1$. If $f$ has bounded partial sums, then there exists a real and primitive Dirichlet character $\chi$ of conductor $q$ such that $\mathbb{D}(f,\chi)<\infty$.
\end{proposition}

\begin{theorem}[Klurman, \cite{klurmancorrelation}, Corollary 3.3]\label{teorema klurman correlations pretentious to 1} Let $f:\NN\to\mathbb{U}$ be a multiplicative function such that $\mathbb{D}(1,f)<\infty$. Let $d\in\NN$. Then
\begin{equation*}
\lim_{x\to \infty}\frac{1}{x}\sum_{n\leq x}f(n)\overline{f(n+d)}=\sum_{r|d}\frac{G(r)}{r},
\end{equation*}
where $G(r)$ is given by:
\begin{equation*}
G(r)=\prod_{\substack{p^k\| r\\(k\geq 0) }} \left(|f\ast\mu(p^k)|^2+2\sum_{i=k+1}^\infty \frac{Re(f\ast\mu(p^k)\overline{f\ast\mu(p^i)} )}{p^{i-k}}\right).
\end{equation*}
\end{theorem}
Now we will suppose that $f:\NN\to\mathbb{U}$ is multiplicative, $\mathbb{D}(f,\chi n^{it})<\infty$ for some $t\in\RR$ and for some primitive Dirichlet character $\chi$ of conductor $q$. We define $F$ to be the multiplicative function such that
\begin{equation}\label{equacao definicao F}
F(p^k)=\begin{cases}f(p^k) \overline{\chi(p^k)}p^{-ikt}, &\mbox{ if }p\nmid q \\
1, &\mbox{ if } p|q. \end{cases}
\end{equation}
For $p\nmid q$, let $F_p(\cdot)$ be the multiplicative function such that for each prime $\tilde{p}$,
\begin{equation*}
F_p(\tilde{p}^k)=\begin{cases} F(p^k), &\mbox{ if }\tilde{p}=p \\
1, &\mbox{ if } \tilde{p}\neq p. \end{cases}
\end{equation*}
For $p\nmid q$, let $M_p(F,\overline{F},d)$ be given by
\begin{equation*}
M_p(F,\overline{F},d)=\lim_{x\to\infty}\frac{1}{x}\sum_{n\leq x}F_p(n)\overline{F_p(n+d)}.
\end{equation*}
\begin{lemma}\label{lemma Klurman auxiliar teorema correlations} Suppose that $p\nmid q$ and that $p^n\|d$, where $n\geq 0$. Let $F$ be given by (\ref{equacao definicao F}). Then
\begin{equation*}
M_p(F,\overline{F},d)=\sum_{a=0}^n\left(\frac{|F\ast\mu(p^a)|^2}{p^a} +2\sum_{i=a+1}^\infty \frac{Re(F\ast\mu(p^a)\overline{F\ast\mu(p^i)} )}{p^{i}} \right).
\end{equation*}
\end{lemma}
\begin{proof} For $p\nmid q$, observe that $\mathbb{D}(1,F_p)=\sqrt{\frac{1-Re(F(p))}{p}}$. Thus Theorem \ref{teorema klurman correlations pretentious to 1} is applicable to $F_p$:
\begin{equation*}
M_p(F,\overline{F},d)=\sum_{r|d}\frac{G_p(r)}{r},
\end{equation*}
where
\begin{equation*}
G_p(r)=\prod_{\substack{\tilde{p}^k\| r\\(k\geq 0) }} \left(|F_p\ast\mu(\tilde{p}^k)|^2+2\sum_{i=k+1}^\infty \frac{Re(F_p\ast\mu(\tilde{p}^k)\overline{F_p\ast\mu(\tilde{p}^i)} )}{\tilde{p}^{i-k}}\right).
\end{equation*}
For $\tilde{p}\neq p$ and $k\geq 1$, we have that $F_p\ast\mu(\tilde{p}^k)=F_p(\tilde{p}^k)-F_p(\tilde{p}^{k-1})=1-1=0$. Thus, if $r$ is divisible by some prime $\tilde{p}\neq p$, for $\tilde{p}^k\| r$, we have that $G_p(r)=0$. Hence, if $p^n\| d$, $n\geq 0$,
\begin{equation*}
M_p(F,\overline{F},d)=\sum_{a=0}^n\frac{G_p(p^a)}{p^a}.
\end{equation*}
For $a\geq 0$:
\begin{align*}
G_p(p^a)&= \left(|F_p\ast\mu(p^a)|^2+2\sum_{i=a+1}^\infty \frac{Re(F_p\ast\mu(p^a)\overline{F_p\ast\mu(p^i)} )}{p^{i-a}}\right)\\
& \times \prod_{\tilde{p}\neq p} \left(|F_p\ast\mu(1)|^2+2\sum_{i=1}^\infty \frac{Re(F_p\ast\mu(1)\overline{F_p\ast\mu(\tilde{p}^i)} )}{\tilde{p}^{i}}\right)\\
&=|F_p\ast\mu(p^a)|^2+2\sum_{i=a+1}^\infty \frac{Re(F_p\ast\mu(p^a)\overline{F_p\ast\mu(p^i)} )}{p^{i-a}}\\
&=|F\ast\mu(p^a)|^2+2\sum_{i=a+1}^\infty \frac{Re(F\ast\mu(p^a)\overline{F\ast\mu(p^i)} )}{p^{i-a}}.
\end{align*}
\end{proof}

\begin{theorem}[Klurman, \cite{klurmancorrelation}, Theorem 1.5, corrected version\footnote{Private communication; The formula (19) of \cite{klurmanchudakov} is correct. In our Lemma \ref{lemma formula klurman} and Appendix \ref{appendix formula klurman} we deduce the correct formulation of Theorem \ref{teorema correlacoes klurman} from formula (19) of \cite{klurmanchudakov}.}]\label{teorema correlacoes klurman} Let $f:\NN\to\mathbb{U}$ be multiplicative and such that $\mathbb{D}(f,\chi n^{it})<\infty$, for some $t\in\RR$ and for some primitive Dirichlet character $\chi$ of conductor $q$. Let $F$ be as in (\ref{equacao definicao F}). For $d\in\NN$, let
\begin{equation*}
S_d:=\lim_{x\to\infty}\frac{1}{x}\sum_{n\leq x}f(n)\overline{f(n+d)}.
\end{equation*}
Then
\begin{equation*}
S_d=\prod_{p\nmid q}M_p(F,\overline{F},d)\prod_{p^l\| q}M_{p^l}(f,\overline{f},d),
\end{equation*}
where
\begin{equation*}
M_{p^l}(f,\overline{f},d)=\begin{cases}0, &\mbox{ if }p^{l-1}\nmid d,\\
-\frac{1}{p},&\mbox{ if }p^{l-1}\| d,\\
\left(1-\frac{1}{p}\right)\sum_{j=0}^k\frac{|f(p^j)|^2}{p^j}-\frac{|f(p^{k+1})|^2}{p^{k+2}}, &\mbox{ if }p^{l+k}\| d, \mbox{ where }k\geq 0,   \end{cases}
\end{equation*}
and if $p^n\| d$ for some $n\geq 0$, $M_p(F,\overline{F},d)$ is given by Lemma \ref{lemma Klurman auxiliar teorema correlations}.
\end{theorem}

\subsection{Proof of Theorem \ref{teorema correlacoes resembling mobius}}
From now on, we will assume that $f:\NN\to\{-1,0,1\}$ is multiplicative, at primes $f(p)=\pm1$, and $f=\mu^2 f$. Further, in light of Proposition \ref{proposicao chi pretentious}, we will assume that $\mathbb{D}(f,\chi)<\infty$ for some real, non-principal and primitive character $\chi$ of conductor $q$.

Let $F$ be given by (\ref{equacao definicao F}). For each prime $p\nmid q$, let
\begin{equation}\label{equacao definicao h}
h(p)=1-\frac{2(1-F(p))}{p}-\frac{2F(p)}{p^2}.
\end{equation}
The hypothesis $\mathbb{D}(f,\chi)<\infty$ implies that the series $\sum_{p\nmid q}|1-h(p)|$ converges,
and hence it is well defined
\begin{equation}\label{equacao definicao C}
C:=\prod_{p\nmid q}h(p).
\end{equation}
Further, a simple calculation shows that $C\neq 0$.
\begin{lemma}\label{lemma auxiliar construcao da v} Let $f$, $F$ and $h$ be as above. Suppose that $p^n\| d$, $n\geq 0$. Then
\begin{equation*}
M_p(F,\overline{F},d)=\begin{cases}h(p),& \mbox{ if }n=0,\\
1-\frac{2}{p^2},& \mbox{ if }n=1,\\
1-\frac{1}{p^2},& \mbox{ if }n\geq 2.  \end{cases}
\end{equation*}
\end{lemma}
\begin{proof} We begin by observing that for each prime $p\nmid q$, and any power $k\geq0$:
\begin{equation*}
F(p^k)=\begin{cases}1,&\mbox{ if }k=0,\\
f(p)\chi(p),&\mbox{ if }k=1, \\
0,&\mbox{ if }k\geq 2.  \end{cases}
\end{equation*}
In particular\footnote{Later we will allow that $f(3)=0$, and hence that $F(3)=0$.}, for $p\nmid q$, $F(p)=\pm 1$. Thus, for $k\geq 1$:
\begin{equation*}
F\ast\mu(p^k)=F(p^k)-F(p^{k-1})=\begin{cases}F(p)-1,&\mbox{ if }k=1,\\
-F(p),&\mbox{ if }k=2, \\
0,&\mbox{ if }k\geq 3.  \end{cases}
\end{equation*}
Now suppose that $p^n\|d$, where $n\geq 0$. If $n=0$, by Lemma \ref{lemma Klurman auxiliar teorema correlations}
\begin{align*}
M_p(F,\overline{F},d)&=\sum_{a=0}^0\left(\frac{|F\ast\mu(p^a)|^2}{p^a} +2\sum_{i=a+1}^\infty \frac{Re(F\ast\mu(p^a)\overline{F\ast\mu(p^i)} )}{p^{i}} \right)\\
&=1+2\sum_{i=1}^\infty \frac{F\ast\mu(p^i)}{p^{i}}=1+2\left(\frac{F(p)-1}{p}-\frac{F(p)}{p^2}  \right)\\
&=h(p).
\end{align*}
If $n=1$:
\begin{align*}
M_p(F,\overline{F},d)&=\sum_{a=0}^1\left(\frac{|F\ast\mu(p^a)|^2}{p^a} +2\sum_{i=a+1}^\infty \frac{Re(F\ast\mu(p^a)\overline{F\ast\mu(p^i)} )}{p^{i}} \right)\\
&=h(p)+\frac{|F\ast\mu(p)|^2}{p} +2\sum_{i=2}^\infty \frac{F\ast\mu(p)F\ast\mu(p^i) }{p^{i}}\\
&=h(p)+\frac{|1-F(p)|^2}{p} +2\frac{(F(p)-1)(-F(p)) }{p^{2}}\\
&=1-\frac{2}{p^2}.
\end{align*}
If $n=2$:
\begin{align*}
M_p(F,\overline{F},d)&=\sum_{a=0}^2\left(\frac{|F\ast\mu(p^a)|^2}{p^a} +2\sum_{i=a+1}^\infty \frac{Re(F\ast\mu(p^a)\overline{F\ast\mu(p^i)} )}{p^{i}} \right)\\
&=1-\frac{2}{p^2}+\frac{|F\ast\mu(p^2)|^2}{p^2} +2\sum_{i=3}^\infty \frac{F\ast\mu(p^2)F\ast\mu(p^i) }{p^{i}}\\
&=1-\frac{2}{p^2}+\frac{1}{p^2}+2\cdot 0\\
&=1-\frac{1}{p^2}.
\end{align*}
If $n\geq 3$:
\begin{align*}
M_p(F,\overline{F},d)&=\sum_{a=0}^n\left(\frac{|F\ast\mu(p^a)|^2}{p^a} +2\sum_{i=a+1}^\infty \frac{Re(F\ast\mu(p^a)\overline{F\ast\mu(p^i)} )}{p^{i}} \right)\\
&=\sum_{a=0}^2\left(\frac{|F\ast\mu(p^a)|^2}{p^a} +2\sum_{i=a+1}^\infty \frac{Re(F\ast\mu(p^a)\overline{F\ast\mu(p^i)} )}{p^{i}} \right)\\
&=1-\frac{1}{p^2}.
\end{align*}
\end{proof}
Let $\NN(q):=\{m\in\NN:\rad(m)|q\}$.
\begin{lemma}\label{lemma construcao da v} Let $h$ be as in \eqref{equacao definicao h} and $v:\NN\to[-1,1]$ be given by
\begin{equation*}
v(d)=\prod_{p\nmid q}M_p(F,\overline{F},d).
\end{equation*}
Then, if $\gcd(d,q)=1$, $v(md)=v(d)$, for all $m\in\NN(q)$. Further, there exists a multiplicative function $g:\NN\to\RR$, such that $v(d)=Cg(d)$, for all $d\in\NN$, where $C$ is given by (\ref{equacao definicao C}) and $g$ is given by:\\
If $p|q$, $g(p^n)=1$; If $p\nmid q$
\begin{equation}\label{equacao definicao da g}
g(p^n)=\begin{cases}\frac{1}{h(p)}\left(1-\frac{2}{p^{2}}\right), &\mbox{ if }n=1,\\
\frac{1}{h(p)}\left(1-\frac{1}{p^{2}}\right),&\mbox{ if }n\geq 2.  \end{cases}
\end{equation}
\end{lemma}
\begin{proof}

Suppose that $\gcd(d,q)=1$. Thus, by Lemma \ref{lemma auxiliar construcao da v}:
\begin{align*}
v(d)&=\prod_{p\nmid q}M_p(F,\overline{F},d)=\prod_{\substack{p\nmid q\\p\nmid d }}M_p(F,\overline{F},d)\prod_{\substack{p^n\|d\\ (n\geq 1) \\  }} M_p(F,\overline{F},d)\\
&=\prod_{\substack{p\nmid q\\p\nmid d }}h(p)\prod_{\substack{p^n\|d\\ (n\geq 1) \\  }} M_p(F,\overline{F},d)=\prod_{\substack{p\nmid q}}h(p)\prod_{\substack{p^n\|d\\ (n\geq 1) \\  }}\frac{1}{h(p)} M_p(F,\overline{F},d)\\
&=Cg(d).
\end{align*}
If $m\in\NN(q)$, since $v(d)$ is given by a product that involves only primes $p\nmid q$, we have that $v(m)=C$. Thus, if $d=m\tilde{d}$, with $m\in\NN(q)$ and $\gcd(\tilde{d},q)=1$, $v(d)=v(\tilde{d})$, and hence, $g(m)=1$ for all $m\in\NN(q)$.
\end{proof}
\begin{lemma}\label{lemma construcao da u} Let $g$ be as in Lemma \ref{lemma construcao da v} and $u=g\ast \mu$. Then we have:
\begin{align*}
a)&\sum_{n=1}^\infty |u(n)|<\infty,\\
b)&\sum_{n\leq H}|u(n)|n=o(H), \\
c)&\sum_{\substack {n\leq H}}u(n)n\ind_{\gcd(n,6)=1}\gg\sqrt{H}.
\end{align*}
 \end{lemma}
\begin{proof}
We have for each prime $p$:
\begin{equation*}
u(p^n)=g(p^n)\mu(1)+g(p^{n-1})\mu(p)=g(p^n)-g(p^{n-1}).
\end{equation*}
If $p|q$, then $g(p^n)=1$ for all $n\geq0$, and hence $u(p^n)=0$ for all $n\geq 1$. If $p\nmid q$,
we have that
\begin{align*}
u(p)&=g(p)-1=\frac{1}{h(p)}\left(1-\frac{2}{p^{2}}\right)-1=\frac{1-h(p)}{h(p)}-\frac{2}{p^2h(p)}\\
&=\frac{1}{h(p)}\left(\frac{2(1-F(p))}{p}+\frac{2F(p)}{p^2}\right)-\frac{2}{p^2h(p)}\\
&=\frac{2(1-F(p))}{h(p)p}-\frac{2(1-F(p))}{p^2h(p)}.
\end{align*}
Further
\begin{align*}
u(p^2)&=g(p^2)-g(p)=\frac{1}{h(p)}\left(1-\frac{1}{p^{2}}\right)-\frac{1}{h(p)}\left(1-\frac{2}{p^{2}}\right)\\
&=\frac{1}{h(p)p^2}.
\end{align*}
For $n\geq 3$, $g(p^n)=g(p^{n-1})$, and hence, $u(p^n)=0$.

To prove that $\sum_{n=1}^\infty |u(n)|$ converges, we only need to show that the series (see \cite{tenenbaumlivro}, pg. 106, Theorem 2)
\begin{equation*}
\sum_{p\in\mathcal{P}}\sum_{m=1}^\infty|u(p^m)|
\end{equation*}
converges. Observe that $h(p)\to 1$ as $p\to\infty$. Hence, the assumption $\mathbb{D}(f,\chi)<\infty$ implies that
\begin{equation*}
\sum_{p\in\mathcal{P}}|u(p)|\leq \sum_{p\nmid q}\left(\frac{2(1-F(p))}{|h(p)|p}+\frac{2(1-F(p))}{p^2|h(p|)}\right)<\infty.
\end{equation*}
Now
\begin{align*}
\sum_{p\in\mathcal{P}}\sum_{m=2}^\infty|u(p^m)|&=\sum_{p\in\mathcal{P}}|u(p^2)|=\sum_{p\nmid q}\frac{1}{|h(p)|p^2}<\infty.
\end{align*}
This shows that the series $\sum_{n=1}^\infty |u(n)|$ converges. In particular, $\sum_{n=1}^\infty \frac{|u(n)|n}{n}$ converges, and hence, by Kroenecker's Lemma (see \cite{shiryaev}, pg. 390 Lemma 2), we have that $\sum_{n\leq H}|u(n)|n=o(H)$.

Now observe that when $F(p)=1$, $h(p)=1-\frac{1}{p^2}>0$ for all primes $p$. When $F(p)=-1$, $h(p)= 1-\frac{4}{p}+\frac{2}{p^2}$. As a function of $p$, $1-\frac{4}{p}+\frac{2}{p^2}$ is increasing for $p>1$, and hence for $p\geq 5$, $h(p)\geq 1-\frac{4}{5}+\frac{2}{5^2}=\frac{7}{25}$. In both cases $F(p)=\pm1$, we have that $h(p)\leq1$ for all $p\geq 5$, and hence $\frac{1}{h(p)}=u(p^2)p^2\geq1$. This implies that $\sum_{\substack {n\leq H}}u(n)n\ind_{\gcd(n,6)=1}\geq \sum_{n\leq H}\ind_{\NN}(\sqrt{n})\mu^2(\sqrt{n})\ind_{\gcd(n,6q)=1}\gg \sqrt{H}$.
\end{proof}
The Lemma below is formula (19) of \cite{klurmanchudakov}, which is implicit in the proof of Theorem \ref{teorema correlacoes klurman}. For the convenience of the reader, we do reverse engineering in Appendix \ref{appendix formula klurman}, \textit{i.e.}, we deduce the formula below from Theorem \ref{teorema correlacoes klurman}.
\begin{lemma}\label{lemma formula klurman} Let $u$ be as in Lemma \ref{lemma construcao da u}. If $\mathbb{D}(f,\chi)<\infty$, where $\chi$ is a primitive Dirichlet character of conductor $q$, then there exists a constant $C=C(f)$ such that
\begin{equation*}
S_d:=\lim_{x\to\infty} \frac{1}{x}\sum_{n\leq x}f(n)\overline{f(n+d)}=\frac{C}{q}\sum_{ \substack{ R|d \\ \rad(R)|q}}\frac{|f(R)|^2}{R}\sum_{a=1}^q \chi(a)\overline{\chi(a+d/R)}\sum_{e|d/R}u(e).
\end{equation*}
\end{lemma}
In our case, $|f(R)|^2=\mu^2(R)$ and $C$ is given by (\ref{equacao definicao C}). A short calculation shows that
\begin{equation*}
\Lambda(H):=\lim_{x\to\infty}\frac{1}{x}\sum_{n\leq x}\bigg{|}\sum_{k=n+1}^{n+H} f(k) \bigg{|}^2=\sum_{|h|\leq H}(H-|h|)S_{|h|},
\end{equation*}
and by Lemma \ref{lemma formula klurman}, after an interchange in the order of summation we obtain that
\begin{equation*}
\Lambda(H)=\frac{C}{q}\sum_{d=1}^\infty u(d)\sum_{\rad(R)|q}\frac{\mu^2(R)}{R}\sum_{\substack{|h|\leq H \\ R|h,\,d|h/R}}(H-|h|)S_\chi(|h|/R),
\end{equation*}
where the inner sum above, is defined and computed as in the same way of \cite{klurmanchudakov} (pg. 684 - 687):
\begin{align*}
&S_\chi(h):=\sum_{a=1}^q\chi(a)\overline{\chi(a+h)},\\
&\sum_{\substack{|h|\leq H \\ R|h,\,d|h/R}}(H-|h|)S_\chi(|h|/R)=qdR\sum_{g|\rad(q)}\frac{\mu(g)}{g^2}\Delta\left(\frac{Hg}{qdR}\right),
\end{align*}
where $\Delta(t)=\{t\}-\{t\}^2$. We thus arrive at:
\begin{equation}\label{equacao Lambda(H)}
\Lambda(H)=C\sum_{d=1}^\infty u(d)d\sum_{\rad(R)|q}\mu^2(R)\sum_{g|\rad(q)}\frac{\mu(g)}{g^2}\Delta\left(\frac{Hg}{qdR}\right).
\end{equation}
Let $\|t\|=\min\{\{t\},1-\{t\}\}$. A short calculation shows that $4\Delta(t)-\Delta(2t)=2\|t\|$. Hence
\begin{equation*}
4\Lambda(qH)-\Lambda(2qH)=2C\sum_{d=1}^\infty u(d)d\sum_{\rad(R)|q}\mu^2(R)\sum_{g|\rad(q)}\frac{\mu(g)}{g^2}\norma\frac{Hg}{dR}\norma.
\end{equation*}
Thus, we have proved:
\begin{lemma}\label{lemma ainda nao sei o nome} Assume that $f$ has bounded partial sums. Then $\Lambda(H)\ll 1$ and hence
\begin{equation*}
S(H):=\sum_{d=1}^\infty u(d)d\sum_{\rad(R)|q}\mu^2(R)\sum_{g|\rad(q)}\frac{\mu(g)}{g^2}\norma\frac{Hg}{dR}\norma\ll 1.
\end{equation*}
\end{lemma}
In Lemma 5.5 of \cite{klurmanchudakov}, it has been proved that for any $t>0$,
\begin{equation}\label{equacao soma positiva da mobius}
\sum_{g|\rad(q)/2^\kappa}\frac{\mu(g)}{g^2}\| gt\|\geq0,
\end{equation}
where $\kappa=1$ if $q$ is even, and $\kappa=0$ otherwise. Now we establish Lemma \ref{lemma quase matador} below whose proof is essentially the one contained in Proposition 5.3 of \cite{klurmanchudakov}. Since the argument is short, we present it for the convenience of the reader:
\begin{lemma}\label{lemma quase matador} Let $\kappa=\ind_{2|q}$. Assume that $f$ has bounded partial sums. Then
\begin{equation*}
\Sigma(H):=\sum_{d=1}^\infty u(d)d\sum_{\rad(R)|q}\mu^2(R)\sum_{g|\rad(q)/2^\kappa}\frac{\mu(g)}{g^2}\norma\frac{Hg}{dR}\norma\ll 1.
\end{equation*}
\end{lemma}
\begin{proof} Let $S(H)$ be as in Lemma \ref{lemma ainda nao sei o nome}. If $q$ is odd, then the result follows from Lemma \ref{lemma ainda nao sei o nome}. Assume then that $q$ is even. We have that
\begin{equation*}
\sum_{g|\rad(q)}\frac{\mu(g)}{g^2}\norma\frac{Hg}{dR}\norma=\sum_{g|\rad(q)/2^\kappa}\frac{\mu(g)}{g^2}\norma\frac{Hg}{dR}\norma-\frac{1}{4}\sum_{g|\rad(q)/2^\kappa}\frac{\mu(g)}{g^2}\norma\frac{Hg}{2dR}\norma,
\end{equation*}
and hence $S(H)=\Sigma(H)-\frac{1}{4}\Sigma(2H)$. Let
\begin{equation*}
T:=\sup_{H}|S(H)|.
\end{equation*}
From Lemma \ref{lemma ainda nao sei o nome}, it is clear that $T<\infty$. We now have for any non-negative integer $K$:
\begin{align*}
\Sigma(H)-\frac{1}{4^K}\Sigma(2^KH)=\sum_{k=0}^{K-1}\frac{1}{4^k}\left(\Sigma(2^kH)-\frac{1}{4}\Sigma(2^{k+1}H)\right)\ll T\sum_{k=0}^{K-1}\frac{1}{4^k}\ll T.
\end{align*}
Now, we recall from Lemma \ref{lemma construcao da u} that $\sum_{d=1}^\infty |u(d)|<\infty$ and $\sum_{n\leq H}|u(n)|n=o(H)$. Observe that for fixed $H$, $d\geq 4H\rad(q)$ implies
\begin{equation*}
\Sigma(H)\ll \sum_{d\leq 4H\rad(q)}|u(d)|d+H\sum_{d\geq 4H\rad(q)}|u(d)|\ll H.
\end{equation*}
In particular, $\frac{1}{4^K}\Sigma(2^KH)\ll \frac{H}{2^K}\to0$ as $K\to\infty$, which concludes the proof.
\end{proof}

\begin{proof}[Proof of Theorem \ref{teorema correlacoes resembling mobius}] From the proof of Lemma \ref{lemma construcao da u}, we have that the function $u$ can be negative only at powers of $2$ or $3$. Thus, we split the proof in 4 possibilities: $F(2),F(3)\geq 0$; $F(2)=-1$ and $F(3)\geq 0$; $F(3)=-1$ and $F(2)\geq 0$; and finally $F(2)=F(3)=-1$. In each of this situations  we will argue by contradiction, \textit{i.e.}, we will suppose that $f$ has bounded partial sums, and hence Lemma \ref{lemma quase matador} holds. By Lemma \ref{lemma construcao da u}, either $u(2),u(4)<0$ or $u(2)=0$ and $u(4)\geq 0$, and similarly, either $u(3),u(9)<0$ or $u(3)=0$ and $u(9)\geq0$. 

\textit{Case 1:}  $u(2)=u(3)=0$. Define
\begin{equation*}
\mathcal{M}(H):=\sum_{\substack{d=1\\ \gcd(d,2)=1}}^\infty u(d)d\sum_{\rad(R)|q}\mu^2(R)\sum_{g|\rad(q)/2^k}\frac{\mu(g)}{g^2}\norma\frac{Hg}{dR}\norma.
\end{equation*}
Since by (\ref{equacao soma positiva da mobius}) the inner sum above is positive, $u(2)=u(3)=0$ and $u(9),u(4)\geq0$, we have that $\mathcal{M}(H)>0$ and $u(4)4\mathcal{M}(H/4)\geq0$, and as $O(1)=\Sigma(H)=\mathcal{M}(H)+u(4)4\mathcal{M}(H/4)$, we have that $\mathcal{M}(H)=O(1)$. By interchanging the order of summation between $\rad(R)|q$ and $d$, the same argument allow us to conclude that
\begin{equation*}
\sum_{\substack{d=1\\ \gcd(d,2)=1}}^\infty u(d)d\sum_{g|\rad(q)/2^k}\frac{\mu(g)}{g^2}\norma\frac{Hg}{d}\norma=O(1).
\end{equation*}
 Let $d_m=2m+1$, where $m\in\{1,2,...,M\}$. Let $H=\frac{1}{2}\lcm(d_1,...,d_M)$. By Lemma \ref{lemma construcao da u}, we then obtain
\begin{equation*}
O(1)=\sum_{\substack{d=1\\ \gcd(d,2)=1}}^\infty u(d)d\sum_{g|\rad(q)/2^k}\frac{\mu(g)}{g^2}\norma\frac{Hg}{d}\norma\geq\frac{1}{2}\prod_{p|q}\left(1-\frac{1}{p^2}\right)\sum_{\substack{d\leq 2M\\ \gcd(d,6)=1}}u(d)d\gg \sqrt{M},
\end{equation*}
which is a contradiction for large $M$. 

\textit{Case 2}: $u(2)<0$ and $u(3)\geq0$ and consequently $u(9)\geq0$. A simple calculation shows that $u(2)2=-4$ and $u(4)4=-2$. We thus have $\Sigma(H)=\mathcal{M}(H)-4\mathcal{M}(H/2)-2\mathcal{M}(H/4)$. Let $T=\sup_H|\Sigma(H)|<\infty$.  Thus for all $H\gg1$
\begin{equation}\label{equacao M(4H)}
\mathcal{M}(4H)=4\mathcal{M}(2H)+2\mathcal{M}(H)+\Sigma(4H)\geq 4\mathcal{M}(2H)+2\mathcal{M}(H)-T.
\end{equation}
As before, by letting $M\gg T^2$ and $H_0=\frac{1}{4}\lcm(d_1,...,d_M)$, for $d\leq 2M$ and $\gcd(dg,2)=1$, we have $\|2H_0g/d\|=1/2$ and $\|H_0g/d\|=1/4$ and hence, we can make $\mathcal{M}(2H_0)\geq 2T$ and $\mathcal{M}(H_0)\geq2 T$. We claim: For all $k\geq 0$
\begin{equation*}
\mathcal{M}(2^k\cdot 4H_0)\geq 4^k\cdot 11T.
\end{equation*}
We prove by induction on $k$. For $k=0$, we have $\mathcal{M}(4H_0)\geq 4\mathcal{M}(2H_0)+2\mathcal{M}(H_0)-T\geq 11T$. Suppose now that the claim holds for all $n\leq k-1$, and consider $n=k\geq1$. By (\ref{equacao M(4H)}):
\begin{align*}
\mathcal{M}(2^k\cdot 4H_0)&\geq 4\mathcal{M}(2^k\cdot 2H_0)+2\mathcal{M}(2^k\cdot H_0)-T\\
&\geq 4\mathcal{M}(2^{k-1}\cdot 4H_0)+2\mathcal{M}(2^{k-2}4\cdot H_0)-T \\
&\geq4\cdot 4^{k-1}11T+4T-T\\
&\geq 4^k\cdot11T.
\end{align*}
Let now $H_k=2^k\cdot 4H_0$. Thus we have proved that $\mathcal{M}(H_k)\geq H_k^2\frac{11T}{16H_0^2}$. However, as we showed in Lemma \ref{lemma quase matador} that $\Sigma(H)\ll H$, in the same line of reasoning we can  show that
$\mathcal{M}(H_k)\ll H_k$, and hence, we again obtain a contradiction. 

\textit{Case 3}: $u(3),u(9)<0$ and $u(2)=0$. In this case $u(4)\geq0$. In this case we define
\begin{equation*}
\mathcal{M}(H):=\sum_{\substack{d=1\\ \gcd(d,3)=1}}^\infty u(d)d\sum_{\rad(R)|q}\mu^2(R)\sum_{g|\rad(q)/2^k}\frac{\mu(g)}{g^2}\norma\frac{Hg}{dR}\norma.
\end{equation*}
Notice that in this case, $u(d)\geq0$ for all $\gcd(d,3)=1$. A simple calculation shows that $u(3)3=-24$ and $u(9)9=-9$. We then have $\Sigma(H)=\mathcal{M}(H)-24\mathcal{M}(H/3)-9\mathcal{M}(H/9)$. In particular for all $H\gg1$:
\begin{equation}\label{equacao M(9H)}
\mathcal{M}(9H)=24\mathcal{M}(3H)+9\mathcal{M}(H)+\Sigma(9H)\geq 24\mathcal{M}(3H)+9\mathcal{M}(H)-T.
\end{equation}

Observe that
\begin{equation*}
\mathcal{M}(H)\geq\sum_{\substack{d=1\\ \gcd(d,3)=1}}^\infty u(d)d\sum_{g|\rad(q)/2^k}\frac{\mu(g)}{g^2}\norma\frac{Hg}{d}\norma.
\end{equation*}
Let $H_0=H_0(M)=\frac{1}{2}\lcm\{d_1,...,d_M\}$, and observe that in the case $u(3)<0$, we necessarily have $\gcd(q,3)=1$, and hence, for $g|\rad(q)/2^\kappa$, $d\leq 2M$ and $\gcd(d,2)=1$, we have $\|3H_0g/d\|,\|H_0g/d\|=\frac{1}{2}$. Thus, again we can make $\mathcal{M}(3H_0),\mathcal{M}(H_0)\gg \sqrt{M}$. We then choose $M$ such that $\mathcal{M}(3H_0),\mathcal{M}(H_0)\geq 10T$. By iterating and doing induction as above, we can conclude that $\mathcal{M}(3^k9H_0)\geq 24^k\cdot 10T$. Define $H_k=3^k9H_0$. We have thus shown that
\begin{equation*}
\mathcal{M}(H_k)\geq \left(\frac{H_k}{9H_0}\right)^{\frac{\log 24}{\log 3}}\cdot10T.
\end{equation*}
Since $\mathcal{M}(H_k)\ll H_k$ and $\frac{\log 24}{\log 3}\geq 2.8$, we again arrive at a contradiction.

\textit{Case 4:} $u(2)<0$ and $u(3)<0$. By Lemma 2.1 of \cite{klurmanteravainen}, we have that for completely multiplicative functions $f,\tilde{f}:\NN\to [-1,1]$ such that $f(p)\neq \tilde{f}(p)$ for only finitely many primes, and for those primes $|\tilde{f}(p)|<1$, then
$$\bigg{|} \sum_{n\leq x}\tilde{f}(n) \bigg{|}\ll\sup_{y\leq x}\bigg{|} \sum_{n\leq y}f(n) \bigg{|}.$$
This is also true if both $f$ and $\tilde{f}$ have support on the squarefree integers. We omit the proof of this since it follows more or less verbatim. Thus, the proof will be completed if $\tilde{f}$ has unbounded partial sums,  where $\tilde{f}(p)$ is equal to $f(p)$ for all primes $p$, except in powers of $3$, where $\tilde{f}(3^k)=0$ for all $k\geq 1$. In this case we can easily compute $u$ related to $\tilde{f}$ and show that  $u(3^k)\geq 0$ for all $k\geq 1$. Thus we can argue exactly as in the case 2 above to conclude that $\tilde{f}$ has unbounded partial sums.
\end{proof}

\subsection{Proof of Theorem \ref{teorema cube free}} We let $f=\mu_2^2g$, where $g:\NN\to\{-1,1\}$ is a completely multiplicative function, and in light of Proposition \ref{proposicao chi pretentious}, we suppose that $\mathbb{D}(f,\chi)<\infty$ for some real and primitive Dirichlet character of conductor $q$. We let $F$ be the multiplicative function given by: For all primes $\gcd(p,q)=1$ and all powers $k$, $F(p^k)=f(p^k)\chi(p^k)$, for each $p|q$ and any power $k$, $F(p^k)=1$. As in the squarefree case, in light of Theorem \ref{teorema correlacoes klurman}, we let
\begin{equation*}
v(d):=\prod_{p\nmid q}M_p(F,\overline{F},d).
\end{equation*}
Then, again, if $\gcd(d,q)=1$, $v(md)=v(d)$, for all $m\in\NN(q)$. For $\gcd(p,q)=1$, we redefine:
\begin{align}\label{equacao definicao h cube free}
h(p)&=1-\frac{2(1-F(p))}{p}+\frac{2(1-F(p))}{p^2}-\frac{2}{p^3},\\
C&=\prod_{p\nmid q}h(p).
\end{align}
As before, we have that $C\neq0$.
\begin{lemma}\label{lemma construcao da u cube free} We have that $v(d)=C(1\ast u)(d)$, where $u$ satisfies the following properties:
\begin{align*}
a)&\sum_{n=1}^\infty |u(n)|<\infty,\\
b)&\sum_{n\leq H}|u(n)|n=o(H), \\
c)&\sum_{\substack {n\leq H}}u(n)n\ind_{\gcd(n,2)=1}\gg H^{1/3}.
\end{align*}
\end{lemma}
\begin{proof}
In what follows $\gcd(p,q)=1$. Then
\begin{equation*}
F(p^k)=\begin{cases}f(p)\chi(p),&\mbox{ if }k=1\\ 1,&\mbox{ if }k=2\\ 0,&\mbox{ if }k\geq 3.\end{cases}
\end{equation*}
As $\mu\ast F(p^k)=F(p^k)-F(p^{k-1})$, we have that
\begin{equation*}
\mu\ast F(p^k)=\begin{cases}F(p)-1,&\mbox{ if }k=1\\ 1-F(p),&\mbox{ if }k=2\\ -1,&\mbox{ if }k= 3\\0,&\mbox{ if }k\geq 4.\end{cases}
\end{equation*}
Suppose that $p^n\| d$. By Lemma \ref{lemma Klurman auxiliar teorema correlations}, we have that
\begin{equation*}
M_p(F,\overline{F},d)=\sum_{a=0}^n\left(\frac{|F\ast\mu(p^a)|^2}{p^a} +2\sum_{i=a+1}^\infty \frac{Re(F\ast\mu(p^a)\overline{F\ast\mu(p^i)} )}{p^{i}} \right).
\end{equation*}
Thus, if $n=0$
\begin{align*}
M_p(F,\overline{F},d)&=\sum_{a=0}^0\left(\frac{|F\ast\mu(p^a)|^2}{p^a} +2\sum_{i=a+1}^\infty \frac{Re(F\ast\mu(p^a)\overline{F\ast\mu(p^i)} )}{p^{i}} \right)\\
&=1+2\left(\frac{F\ast \mu(p)}{p}+\frac{F\ast \mu(p^2)}{p^2}+\frac{F\ast \mu(p^3)}{p^3} \right)\\
&=1+2\left(\frac{F(p)-1}{p}+\frac{1-F(p)}{p^2}-\frac{1}{p^3}\right)\\
&=h(p).
\end{align*}
If $n=1$,
\begin{align*}
M_p(F,\overline{F},d)&=\sum_{a=0}^1\left(\frac{|F\ast\mu(p^a)|^2}{p^a} +2\sum_{i=a+1}^\infty \frac{Re(F\ast\mu(p^a)\overline{F\ast\mu(p^i)} )}{p^{i}} \right)\\
&=h(p)+\frac{|F\ast\mu(p)|^2}{p}+2F\ast\mu(p)\left(\frac{F\ast \mu(p^2)}{p^2}+\frac{F\ast \mu(p^3)}{p^3} \right)\\
&=h(p)+\frac{|F(p)-1|^2}{p}+2(F(p)-1)\left(\frac{1-F(p)}{p^2}-\frac{1}{p^3} \right)\\
&=1-\frac{2(1-F(p))}{p^2}-\frac{2F(p)}{p^3}.
\end{align*}
If $n=2$,
\begin{align*}
M_p(F,\overline{F},d)&=\sum_{a=0}^2\left(\frac{|F\ast\mu(p^a)|^2}{p^a} +2\sum_{i=a+1}^\infty \frac{Re(F\ast\mu(p^a)\overline{F\ast\mu(p^i)} )}{p^{i}} \right)\\
&=1-\frac{2(1-F(p))}{p^2}-\frac{2F(p)}{p^3}+\frac{|F\ast \mu(p^2)|^2}{p^2}+2F\ast\mu(p^2)\frac{F\ast\mu(p^3)}{p^{3}} \\
&=1-\frac{2(1-F(p))}{p^2}-\frac{2F(p)}{p^3}+\frac{|1-F(p)|^2}{p^2}+2(1-F(p))\frac{-1}{p^{3}}\\
&=1-\frac{2}{p^3}.
\end{align*}
If $n=3$,
\begin{align*}
M_p(F,\overline{F},d)&=\sum_{a=0}^3\left(\frac{|F\ast\mu(p^a)|^2}{p^a} +2\sum_{i=a+1}^\infty \frac{Re(F\ast\mu(p^a)\overline{F\ast\mu(p^i)} )}{p^{i}} \right)\\
&=1-\frac{2}{p^3}+\frac{|F\ast\mu(p^3)|^2}{p^3}=1-\frac{1}{p^3}.
\end{align*}
If $n\geq 4$
\begin{align*}
M_p(F,\overline{F},d)&=\sum_{a=0}^3\left(\frac{|F\ast\mu(p^a)|^2}{p^a} +2\sum_{i=a+1}^\infty \frac{Re(F\ast\mu(p^a)\overline{F\ast\mu(p^i)} )}{p^{i}} \right)\\
&=1-\frac{1}{p^3}.
\end{align*}
As
\begin{align*}
v(d)&=\prod_{p\nmid q}M_p(F,\overline{F},d)=\prod_{\substack{p\nmid d \\ p\nmid q}}M_p(F,\overline{F},d)\prod_{\substack{p^n\|d \\ (n\geq 1)}}M_p(F,\overline{F},d) \\
&=\prod_{\substack{p\nmid q}}h(p)\prod_{\substack{p^n\|d\\(n\geq 1)}}\frac{1}{h(p)}M_p(F,\overline{F},d),
\end{align*}
we obtain that $v(d)=Cg(d)$, where $g$ is the multiplicative function given by $g(p^n)=1$ if $p|q$, and if $\gcd(p,q)=1$:
\begin{equation*}
g(p^n)=\begin{cases}\frac{1}{h(p)}\left(1-\frac{2(1-F(p))}{p^2}-\frac{2F(p)}{p^3} \right),&\mbox{ if }n=1,\\ \frac{1}{h(p)}\left(1-\frac{2}{p^3} \right), &\mbox{ if }n=2\\
\frac{1}{h(p)}\left(1-\frac{1}{p^3}\right), &\mbox{ if }n\geq 3.\end{cases}
\end{equation*}
Let $u=\mu\ast g$. Then $u(p^k)=g(p^k)-g(p^{k-1})$. Thus, clearly $u(p^k)=0$ if $k\geq 1$ and $p|q$. Now, if $\gcd(p,q)=1$, we have that $u(p^k)=0$ for $k\geq 4$, and for $k\leq 3$:
\begin{align*}
u(p)&=g(p)-1=\frac{1}{h(p)}\left(1-\frac{2(1-F(p))}{p^2}-\frac{2F(p)}{p^3} -h(p)\right)\\
&=\frac{1}{h(p)}\left(1-\frac{2(1-F(p))}{p^2}-\frac{2F(p)}{p^3}-\left(1-\frac{2(1-F(p))}{p}+\frac{2(1-F(p))}{p^2}-\frac{2}{p^3}\right)\right)\\
&=\frac{1}{h(p)}\left(\frac{2(1-F(p)}{p}-\frac{4(1-F(p))}{p^2}+\frac{2(1-F(p))}{p^3}\right) \\
&=\frac{2(1-F(p)}{h(p)}\left(\frac{1}{p}-\frac{2}{p^2}+\frac{1}{p^3}\right).
\end{align*}
\begin{align*}
u(p^2)&=\frac{1}{h(p)}\left(1-\frac{2}{p^3}-\left(1-\frac{2(1-F(p))}{p^2}-\frac{2F(p)}{p^3} \right)\right)\\
&=\frac{2(1-F(p))}{h(p)p^2}\left(1-\frac{1}{p}\right).
\end{align*}
\begin{align*}
u(p^3)&=\frac{1}{h(p)}\left(1-\frac{1}{p^3}-\left(1-\frac{2}{p^3} \right)\right)\\
&=\frac{1}{h(p)p^3}.
\end{align*}
Thus, we have that the sign of $u(p^l)$ is dictated by $h(p)$. Since
\begin{align*}
h(p)&=1-\frac{2(1-F(p))}{p}+\frac{2(1-F(p))}{p^2}-\frac{2}{p^3},\\
\end{align*}
we have that $0<h(p)<1$ in the case that $F(p)=1$. In the case $F(p)=-1$, we have that
\begin{align*}
h(p)&=1-\frac{4}{p}+\frac{4}{p^2}-\frac{2}{p^3}=\frac{p^3-4p^2+4p-2}{p^3}\leq \frac{p^3-2}{p^3}<1.\\
\end{align*}
In this case, we have that $h(2)=-1/4$. Let $s(p)=p^3-4p^2+4p-2$. Then $s'(p)=3p^2-8p+4=3(p-2)(p-2/3)$, and hence $s(p)$ is increasing for $p\geq 3$. Hence, for $p\geq 3$, $s(p)\geq s(3)=1>0$. In particular, in any case we have $h(p)>0$ and $u(p^3)p^3\geq 1$ for all $p\geq 3$. Similarly to Lemma \ref{lemma construcao da u}, we can show that $\sum_{n=1}^\infty |u(n)|<\infty$ and that $\sum_{n\leq H}|u(n)|n=o(H)$. Let $\tilde{u}$ be the multiplicative function such that $\tilde{u}(p^k)=\ind_{k=3}\ind_{\gcd(p,2q)=1}$.
Thus $\tilde{U}(s):=\sum_{n=1}^\infty\frac{\tilde{u}(n)}{n^s}=\prod_{p\nmid 2q}(1+p^{-3s})$. Thus the Dirichlet series $\tilde{U}(s)\zeta^{-1}(3s)=\prod_{p|2q}(1-p^{-3s})\prod_{p\nmid 2q}(1-p^{6s})$ converges absolutely for all $Re(s)>1/6$. Hence $\sum_{n\leq H}\tilde{u}(n)\gg H^{1/3}$.  Now observe that
\begin{equation*}
\sum_{\substack {n\leq H}}u(n)n\ind_{\gcd(n,2)=1}\geq\sum_{n\leq H}\tilde{u}(n)\gg H^{1/3}.
\end{equation*}
\end{proof}
\begin{proof}[Proof of Theorem \ref{teorema cube free}] Let $f=\mu_2^2g$, where $g:\NN\to\{-1,1\}$ is a completely multiplicative function. Similarly as in the squarefree case, we can show that if $f$  has bounded partial sums, then
\begin{equation}\label{equacao Sigma cube free}
\Sigma(H):=\sum_{d=1}^\infty u(d)d\sum_{\rad(R)|q}\mu_2^2(R)\sum_{g|\rad(q)/2^\kappa}\frac{\mu(g)}{g^2}\norma\frac{Hg}{dR}\norma\ll 1,
\end{equation}
where $\kappa=1$ if $q$ is even, or $0$ otherwise. Recall that the inner sum $\sum_{g|\rad(q)/2^\kappa}$ is non-negative. By Lemma \ref{lemma construcao da u cube free} the function $u$ can be negative only at powers of 2, and either $u(2^l)\geq 0$ for $l=1,2,3$, or $u(2^l)<0$ for $l=1,2,3$. Consider first the case that all $u(2^l)\geq0$.
In this case
\begin{equation*}
\Sigma(H)\geq\sum_{d=1}^\infty u(d)d\sum_{g|\rad(q)/2^\kappa}\frac{\mu(g)}{g^2}\norma\frac{Hg}{d}\norma.
\end{equation*}
Let $d_1,...,d_M$ be the first $M$ odd integers, and let $H=\frac{1}{2} \lcm(d_1,...,d_M)$. Thus for each $d_i$, we have that $\|Hg/d\|=1/2$. Hence, by Lemma \ref{lemma construcao da u cube free},
\begin{equation*}
\Sigma(H)\gg\sum_{\substack{d\leq 2M\\ \gcd(d,2)=1}}u(d)d\gg M^{1/3},
\end{equation*}
which is a contradiction for sufficiently large $M$. Suppose now that $u(2)<0$. A simple calculation shows that $u(2)2=-4$, $u(4)4=-8$ and $u(8)8=-4$. Let
\begin{equation*}
\MM(H)=\sum_{\substack{d=1\\ \gcd(d,2)=1}}^\infty u(d)d\sum_{\rad(R)|q}\mu_2^2(R)\sum_{g|\rad(q)/2^\kappa}\frac{\mu(g)}{g^2}\norma\frac{Hg}{dR}\norma.
\end{equation*}
Then $\Sigma(H)=\MM(H)-4\MM(H/2)-8\MM(H/4)-4\MM(H/8)$.
Let $T=\sup_H|\Sigma(H)|$. Then for all $H\gg1$
\begin{equation}\label{equacao recursao cube free}
\MM(4H)\geq4\MM(2H)+8\MM(H)+4\MM(H/2)-T\geq 4\MM(2H)+8\MM(H)-T.
\end{equation}
let $H_0=H_0(M)=\frac{1}{4} \lcm(d_1,...,d_M)$. Since $g|\rad(q)/2^\kappa$, $\|2H_0g/d_i\|=1/2$. On the other hand write $\lcm(d_1,...,d_M)g/d_i=2l+1$. Then $\|H_0g/d_i\|=\|(2l+1)/4\|=1/4$ if $l$ is even, otherwise, if $l=2l'+1$, $\|(2l+1)/4\|=\|(2l'+1)/2+1/4\|=\|3/4\|=1/4$. Thus,
\begin{equation*}
\MM(2H_0)\geq \sum_{\substack{d=1\\ \gcd(d,2)=1}}^\infty u(d)d\sum_{g|\rad(q)/2^\kappa}\frac{\mu(g)}{g^2}\norma\frac{2H_0g}{d}\norma\gg \sum_{\substack{d\leq 2M\\ \gcd(d,2)=1}}^\infty u(d)d\gg M^{1/3}.
\end{equation*}
Hence, by making $M\gg T^3$, we can make $\MM(2H_0),\MM(H_0)\geq 10 T$. Iterating and using induction in (\ref{equacao recursao cube free}), we conclude that $\mathcal{M}(2^kH_0)\geq 4^k2T$. Defining $H_k=2^kH_0$, we obtain that $\MM(H_k)\gg H_k^2$. However, as in the proof of Theorem \ref{teorema correlacoes resembling mobius}, we have that $\MM(H_k)\ll H_k$, and hence we again obtain a contradiction.
\end{proof}

\subsection{Proof of Theorem \ref{teorema H^1/4 cancelacao}}
\begin{proof} Consider first the squarefree case. Since $f_1(p)=\chi(p)$ for all $\gcd(p,q)=1$, we have, by Lemma \ref{lemma construcao da u}, that $u(d)d$ is always non-negative, and moreover, it is supported on the perfect squares coprime with $q$, and $u(p^2)p^2=\frac{1}{h(p)}$, where $h(p)=1-\frac{2}{p^2}$.
Consider, then, the Dirichlet series $U(s)=\sum_{d=1}^\infty \frac{u(d)d}{d^s}$. Let $F(s)=\sum_{n=1}^\infty \frac{a(n)}{n^s}:= U(s)\zeta(2s)^{-1}$. Thus
\begin{equation*}
F(s)=\prod_{p|q}\left(1-\frac{1}{p^{2s}}\right)\prod_{p\nmid q}\left( 1+\frac{2}{p^2 h(p)p^{2s}}-\frac{1}{h(p)p^{4s}}\right).
\end{equation*}
Clearly, from the Euler product above, we have that $F(s)$ converges absolutely for all $Re(s)>1/4$, and hence $\sum_{n\leq x}a(n)\ll x^{1/4+\epsilon}$, for any $\epsilon>0$. Now define $\zeta(2s):=\sum_{n=1}^\infty \frac{\ind^{(2)}(n)}{n^s}$. Then $\sum_{n\leq x}\ind^{(2)}(n)=\sum_{n\leq \sqrt{x}}1=[\sqrt{x}]$.
Thus, since $u(d)d=\ind^{(2)}\ast a(d)$, we have
\begin{align*}
\sum_{d\leq H}u(d)d=\sum_{d\leq H}a(d)\bigg{[}\sqrt{\frac{H}{d}}\bigg{]}=\sqrt{H}\sum_{d=1}^\infty \frac{a(d)}{\sqrt{d}}+O(H^{1/4+\epsilon}):=D\sqrt{H}+O(H^{1/4+\epsilon}).
\end{align*}
On the other hand, we have that
\begin{align*}
H\sum_{d>H}u(d)&=H\int_{H^+}^\infty\frac{1}{x}d\left(\sum_{n\leq x}u(n)n\right)=-\sum_{d\leq H}u(d)d+H\int_{H}^\infty \left(\sum_{n\leq x}u(n)n\right)\frac{dx}{x^2}\\
&=-D\sqrt{H}+2D\sqrt{H} +O(H^{1/4+\epsilon})=D\sqrt{H}+O(H^{1/4+\epsilon}).
\end{align*}
Now recall formula (\ref{equacao Lambda(H)}):
\begin{equation*}
\Lambda(H)=C\sum_{d=1}^\infty u(d)d\sum_{\rad(R)|q}\mu^2(R)\sum_{g|\rad(q)}\frac{\mu(g)}{g^2}\Delta\left(\frac{Hg}{qdR}\right),
\end{equation*}
where $\Delta(t)=\{t\}-\{t\}^2\leq \{t\}$. As $\sum_{\rad(R)|q}\mu^2(R)<\infty$, we split the sum $\sum_{d=1}^\infty$ into $\sum_{d\leq \alpha H}$ and $\sum_{d> \alpha H}$, where $\alpha$ is such that: If $d>\alpha H$, then $\frac{Hg}{qdR}<1$, for all $g|\rad(q)$. Thus we obtain that
\begin{equation*}
\Lambda(H)\ll \sum_{d\leq \alpha H}u(d)d+H\sum_{d>\alpha H}u(d)\ll \sqrt{H}.
\end{equation*}
Now we consider the cubefree case. Since $f_2(p)=\chi(p)$ for all $\gcd(p,q)=1$, by Lemma \ref{lemma construcao da u cube free}, we have that $u$ has support on the perfect cubes coprime with $q$, and $u(p^3)p^3=\frac{1}{h(p)}$, where $h(p)=1-\frac{2}{p^3}$. Thus we can proceed  in the same way as above. \end{proof}

\appendix
\section{}
\subsection{Proposition \ref{proposicao chi pretentious}}\label{appendix proposicao chi pretentious}
\begin{proof} Let $A(x)=\sum_{n\leq x}f^2(n)$. Let $y$ be large such that $A(t)\geq ct/2$ for all $t\geq y$. Then, for $x\geq y^{4/c}$:
\begin{align*}
\sum_{n\leq x}\frac{f^2(n)}{n}&=\int_{1^-}^xt^{-1}dA(t)=x^{-1}A(x)+\int_{1}^xt^{-2}A(t)dt\\
&\geq \int_{y}^\infty t^{-2}A(t)dt\geq \frac{c}{2}(\log x-\log y)\\
&\geq\frac{c}{4}\log x.
\end{align*}
Consider
\begin{equation*}
\Lambda(H)=\frac{1}{\log x}\sum_{n\leq x}\frac{1}{n}\left(\sum_{k=n+1}^{n+H}f(k)\right)^2.
\end{equation*}
Expanding the square above, we have that the diagonal contribution is
\begin{align*}
\frac{1}{\log x}\sum_{n\leq x}\frac{1}{n}\sum_{k=n+1}^{n+H}f^2(k)&=\frac{1}{\log x}\sum_{n\leq x}\frac{1}{n}\sum_{k=1}^{H}f^2(n+k)=\frac{1}{\log x}\sum_{k=1}^{H}\sum_{n\leq x}\frac{f^2(n+k)}{n}\\
&=\frac{1}{\log x}\sum_{k=1}^{H}\sum_{n\leq x}\frac{f^2(n+k)}{n+k}+O\left(\frac{H^2}{\log x}\right)\\
&=\frac{H}{\log x}\sum_{1\leq n\leq x}\frac{f^2(n)}{n}+O\left(\frac{H^2}{\log x}\right)\\
&\geq \frac{cH}{4}+O\left(\frac{H^2}{\log x}\right),
\end{align*}
provided that $x$ is large enough. Suppose now that the partial sums of $f$ are bounded by $C$. Hence, by the triangle inequality, $\Lambda(H)\leq 8 C^2$. Let
\begin{equation*}
S_h=S_h(x):=\frac{1}{\log x}\sum_{n\leq x}\frac{f(n)f(n+h)}{n}.
\end{equation*}
Suppose, by contradiction, that for fixed $H$, there exists arbitrarily large values of $x>0$, such that for each $h\leq H$, $|S_h|\leq \frac{c}{20H}$. Thus, for these values of $x$, the non-diagonal terms in $\Lambda(H)$ are
\begin{align*}
&\frac{2}{\log x}\sum_{n\leq x}\frac{1}{n}\sum_{k=n+1}^{n+H}\sum_{l=k+1}^{n+H}f(k)f(l)=\frac{2}{\log x}\sum_{n\leq x}\frac{1}{n}\sum_{k=n+1}^{n+H}\sum_{h=1}^{n+H-k}f(k)f(k+h)\\
&=\frac{2}{\log x}\sum_{n\leq x}\frac{1}{n}\sum_{k=1}^{H}\sum_{h=1}^{H-k}f(n+k)f(n+k+h)\\
&=\frac{2}{\log x}\sum_{h=1}^H\sum_{k=1}^{H-h}\sum_{n\leq x}\frac{f(n+k)f(n+k+h)}{n}\\
&=\frac{2}{\log x}\sum_{h=1}^{H-1}\sum_{k=1}^{H-h}\sum_{n\leq x}\frac{f(n+k)f(n+k+h)}{n+k}+O\left(\frac{H^2\log H}{\log x}\right)\\
&=\frac{2}{\log x}\sum_{h=1}^{H-1}\sum_{k=1}^{H-h}\sum_{n\leq x}\frac{f(n)f(n+h)}{n}+O\left(\frac{H^2\log H}{\log x}+\frac{H^3}{x\log x}\right)\\
&=2\sum_{h=1}^{H}(H-h)S_h+O\left(\frac{H^2\log H}{\log x}+\frac{H^3}{x\log x}\right)\\
&\geq -2\frac{c}{20H}\sum_{h=1}^{H-1}h+O\left(\frac{H^2\log H}{\log x}+\frac{H^3}{x\log x}\right)\\
&\geq -\frac{cH}{20}+O\left(\frac{H^2\log H}{\log x}+\frac{H^3}{x\log x}\right).
\end{align*}
Thus,
\begin{equation*}
\Lambda(H)\geq \frac{cH}{5}+O\left(\frac{H^2\log H}{\log x}+\frac{H^3}{x\log x}\right),
\end{equation*}
and selecting $H\geq500 C^2/c$, we obtain a contradiction for sufficiently large $x$. Thus, for fixed $H\geq 500 C^2/c$, for all $x>0$ sufficiently large, there exists $h_x\leq H$ such that $|S_{h_x}|\geq \frac{c}{20H}$. Thus, for a sequence $x_n\to\infty$, there exists a fixed $h\leq H$, such that $S_h(x_n)\geq \frac{c}{20H}$. Now, as in Lemma 4.3 of \cite{klurmancorrelation}, or Proposition 1.11 of \cite{taodiscrepancy}, we conclude that there exists $A=A(H)\geq 0$ such that for any sufficiently large $x$, there exists $t_x\leq Ax$ and a primitive character $\chi$ of modulus $D\leq A$ such that
$\mathbb{D}(f(n),\chi(n)n^{it_x};x)\leq A$. Now, as in Lemma 2.5 of \cite{klurmanchudakov}, we conclude that there is a primitive character $\chi$ of modulus $q$ and a real number $t$ such that $\mathbb{D}(f(n),\chi(n)n^{it})<\infty$. Since for each prime $p$, $f(p)^2=1$, we can argue as in the same line of reasoning of Lemma 5.1 of \cite{klurmanchudakov} to conclude that $\mathbb{D}(f(n),\chi(n))<\infty$. Finally, $\chi$ must be real. To see that this is true, suppose that $\chi$ is not real. Then we can select $1\leq k \leq q$, $\gcd(k,q)=1$, and $|Re(\chi(k))|=|\cos(\theta)|<1$. For each prime $\gcd(p,q)=1$, let $\chi(p)=e^{i\theta p}$. Thus, $1-Re(f(p)\overline{\chi(p)})=1-f(p)\cos(\theta_p)\geq 1-|\cos(\theta_p)|$.
Thus, by standard estimates for the number of primes in arithmetic progressions, we obtain a contradiction: $\mathbb{D}(f,\chi)^2\geq \sum_{p\in\mathcal{P}}\frac{1-|\cos(\theta_p)|}{p}\geq (1-|\cos(\theta)|)\sum_{p\equiv k\mod q}\frac{1}{p}=\infty$.  \end{proof}
\subsection{Lemma \ref{lemma formula klurman}}\label{appendix formula klurman}
\begin{proof}
Let $d=md'$, where $\gcd(d',q)=1$ and $m\in\NN(q)$. Thus, for any function $\lambda:\NN\to\CC$, to sum $\lambda(R)$ over $R|d$ such that  $\rad(R)|q$ is equal to the sum of $\lambda(R)$ over those $R|m$. If $R|m$, we have that $\frac{d}{R}=\frac{m}{R}d'$, and $\frac{m}{R}\in\NN(q)$. Notice that $u(e)=0$ whenever $\gcd(e,q)>1$. Hence, $\sum_{e|d/R}u(e)=\sum_{e|d'm/R}u(e)=\sum_{e|d'}u(e)=(1\ast u)(d')=\frac{1}{C}v(d')$. Next, we recall that, by Lemma \ref{lemma construcao da v}, $v(d')=v(md')=v(d)$. Thus,
\begin{equation*}
\frac{C}{q}\sum_{R|m}\frac{|f(R)|^2}{R}\sum_{a=1}^q \chi(a)\overline{\chi(a+d/R)}\sum_{e|d/R}u(e)=\frac{v(d)}{q}\sum_{R|m}\frac{|f(R)|^2}{R}\sum_{a=1}^q \chi(a)\overline{\chi(a+d/R)}.
\end{equation*}
In \cite{klurmanchudakov}, pg. 685, $\sum_{a=1}^q \chi(a)\overline{\chi(a+md'/R)}$ is calculated with great detail:
\begin{equation*}
\sum_{a=1}^q \chi(a)\overline{\chi(a+md'/R)}=q\prod_{p^l\|q}\left(\ind_{p^l|m/R}-\frac{1}{p}\ind_{p^{l-1}|m/R}\right).
\end{equation*}
Thus, we wish to prove that
\begin{equation*}
\sum_{  R|m }\frac{|f(R)|^2}{R}\prod_{p^l\|q}\left(\ind_{p^l|m/R}-\frac{1}{p}\ind_{p^{l-1}|m/R}\right)=\prod_{p^l\| q}M_{p^l}(f,\overline{f},m),
\end{equation*}
where $M_{p^l}(f,\overline{f},m)$ is given by Theorem \ref{teorema correlacoes klurman}. We are going to do that by induction on the number $\omega(q)$. Suppose, then, that $\omega(q)=1$, \textit{i.e}, $q=p^l$ for some prime $p$ and some power $l\geq1$. We have, then, that $m=p^\alpha$ for some power $\alpha\geq 0$. Thus, if $\alpha\leq l-2$:
\begin{align*}
\sum_{  R|p^\alpha }\frac{|f(R)|^2}{R}\left(\ind_{p^l|p^\alpha/R}-\frac{1}{p}\ind_{p^{l-1}|p^\alpha/R}\right)=0.
\end{align*}
Now, if $\alpha=l-1$:
\begin{align*}
\sum_{  R|p^\alpha }\frac{|f(R)|^2}{R}\left(\ind_{p^l|p^\alpha/R}-\frac{1}{p}\ind_{p^{l-1}|p^\alpha/R}\right)=-\frac{1}{p}.
\end{align*}
Now, if $\alpha=l+k$, $k\geq 0$:
\begin{align*}
\sum_{R|p^{l+k} }\frac{|f(R)|^2}{R}\left(\ind_{R|p^k}-\frac{1}{p}\ind_{R|p^{k+1}}\right)=\left(1-\frac{1}{p}\right)\sum_{j=0}^k\frac{|f(p^j)|^2}{p^j}-\frac{|f(p^{k+1})|^2}{p^{k+1}}.
\end{align*}
Thus we have proved the statement in the case $\omega(q)=1$. Suppose then that the statement is true for all $q$ with $\omega(q)\leq n$. Let $q=q'\tilde{p}^{\tilde{l}}$, where $\gcd(\tilde{p},q')=1$ and $\omega(q')=n$. We have, then, that $m=m'\tilde{p}^\alpha$, where $m'\in\NN(q')$ and $\alpha\geq0$. Hence,
\begin{equation*}
\sum_{R|m}=\sum_{\beta=0}^\alpha\sum_{\substack{\tilde{p}^\beta R|m\\ \gcd(R,\tilde{p})=1}}=\sum_{\beta=0}^\alpha\sum_{\substack{\tilde{p}^\beta R|m\\ R|m'}}.
\end{equation*}
Hence
\begin{align*}
&\sum_{  R|m }\frac{|f(R)|^2}{R}\prod_{p^l\|q}\left(\ind_{p^l|m/R}-\frac{1}{p}\ind_{p^{l-1}|m/R}\right)\\
=&\sum_{\beta=0}^\alpha \sum_{\substack{\tilde{p}^\beta R|m'\tilde{p}^\alpha\\ R|m'}}\frac{|f(\tilde{p}^\beta R)|^2}{\tilde{p}^\beta R}\prod_{p^l\|q'}\left(\ind_{p^l|m'\tilde{p}^\alpha/\tilde{p}^\beta R}-\frac{1}{p}\ind_{p^{l-1}|m'\tilde{p}^\alpha/\tilde{p}^\beta R}\right)\left(\ind_{\tilde{p}^{\tilde{l}}|m'\tilde{p}^\alpha/\tilde{p}^\beta R}-\frac{1}{p}\ind_{\tilde{p}^{\tilde{l}-1}|m'\tilde{p}^\alpha/\tilde{p}^\beta R}\right)\\
=&\sum_{\beta=0}^\alpha \sum_{R|m'}\frac{|f(\tilde{p}^\beta R)|^2}{\tilde{p}^\beta R}\prod_{p^l\|q'}\left(\ind_{p^l|m'/R}-\frac{1}{p}\ind_{p^{l-1}|m'/R}\right)\left(\ind_{\tilde{p}^{\tilde{l}}|\tilde{p}^\alpha/\tilde{p}^\beta }-\frac{1}{p}\ind_{\tilde{p}^{\tilde{l}-1}|\tilde{p}^\alpha/\tilde{p}^\beta }\right)\\
=&\sum_{\beta=0}^\alpha \frac{|f(\tilde{p}^\beta)|^2}{\tilde{p}^\beta} \left(\ind_{\tilde{p}^{\tilde{l}}|\tilde{p}^\alpha/\tilde{p}^\beta }-\frac{1}{p}\ind_{\tilde{p}^{\tilde{l}-1}|\tilde{p}^\alpha/\tilde{p}^\beta }\right)\sum_{R|m'}\frac{|f(R)|^2}{R}\prod_{p^l\|q'}\left(\ind_{p^l|m'/R}-\frac{1}{p}\ind_{p^{l-1}|m'/R}\right)\\
=&\sum_{\beta=0}^\alpha \frac{|f(\tilde{p}^\beta)|^2}{\tilde{p}^\beta} \left(\ind_{\tilde{p}^{\tilde{l}}|\tilde{p}^\alpha/\tilde{p}^\beta }-\frac{1}{p}\ind_{\tilde{p}^{\tilde{l}-1}|\tilde{p}^\alpha/\tilde{p}^\beta }\right)\prod_{p^l\|q'}M_{p^l}(f,\overline{f},m')\\
=&M_{\tilde{p}^{\tilde{l}}}(f,\overline{f},\tilde{p}^\alpha)\prod_{p^l\|q'}M_{p^l}(f,\overline{f},m'),
\end{align*}
where in the penultimate equality we used the induction hypothesis in the case $\omega(q')=n$, and in the last equality we used the induction in the case that $\omega(\tilde{p}^l)=1$. Finally, we notice that,
for each $p|q'$, $M_{p^l}(f,\overline{f},m')=M_{p^l}(f,\overline{f},m'\tilde{p}^\alpha)=M_{p^l}(f,\overline{f},m)=M_{p^l}(f,\overline{f},md')=M_{p^l}(f,\overline{f},d)$. Similarly, $M_{\tilde{p}^{\tilde{l}}}(f,\overline{f},p^\alpha)=M_{\tilde{p}^{\tilde{l}}}(f,\overline{f},d)$.
\end{proof}

{\small{\sc \noindent Marco Aymone \\
Departamento de Matem\'atica, Universidade Federal de Minas Gerais (UFMG), Brazil.}\\
\textit{Email address:} aymone.marco@gmail.com}


\begin{thebibliography}{1}

\bibitem{aymoneresemblingmobius}
{\sc M.~Aymone}, {\em A note on multiplicative functions resembling the
  {M}\"{o}bius function}, J. Number Theory, 212 (2020), pp.~113--121.

\bibitem{aymonekfree}
{\sc M.~Aymone, C. Bueno and K. Medeiros}, {\em Multiplicative functions supported on the $k$-free integers with small partial sums}, arXiv:2101.00279, 2021



\bibitem{granvillepretentious}
{\sc A.~Granville and K.~Soundararajan}, {\em Pretentious multiplicative
  functions and an inequality for the zeta-function}, in Anatomy of integers,
  vol.~46 of CRM Proc. Lecture Notes, Amer. Math. Soc., Providence, RI, 2008,
  pp.~191--197.



\bibitem{klurmancorrelation}
{\sc O.~Klurman}, {\em Correlations of multiplicative functions and
  applications}, Compos. Math., 153 (2017), pp.~1622--1657.

\bibitem{klurmanchudakov}
{\sc O.~Klurman and A.~P. Mangerel}, {\em Rigidity theorems for multiplicative
  functions}, Math. Ann., 372 (2018), pp.~651--697.

\bibitem{klurmanteravainen}
{\sc O.~Klurman, A.~P. Mangerel, C.~Pohoata, and J.~Ter\"av\"ainen}, {\em
  Multiplicative functions that are close to their mean}, Trans. of the Amer. Math. Soc., to appear
  (2021).

\bibitem{shiryaev}
{\sc A.~N. Shiryaev}, {\em Probability}, vol.~95 of Graduate Texts in
  Mathematics, Springer-Verlag, New York, second~ed., 1996.
\newblock Translated from the first (1980) Russian edition by R. P. Boas.

\bibitem{taodiscrepancy}
{\sc T.~Tao}, {\em The {E}rd\H{o}s discrepancy problem}, Discrete Anal.,
  (2016), pp.~Paper No. 1, 29.

\bibitem{tenenbaumlivro}
{\sc G.~Tenenbaum}, {\em Introduction to analytic and probabilistic number
  theory}, vol.~46 of Cambridge Studies in Advanced Mathematics, Cambridge
  University Press, Cambridge, 1995.
\newblock Translated from the second French edition (1995) by C. B. Thomas.

\end{thebibliography}
\end{document}